\documentclass[10pt,reqno]{amsart}

\usepackage{amsmath,amssymb,amsfonts,latexsym,amsthm}

\usepackage{mathrsfs}

\numberwithin{equation}{section}

\usepackage{graphicx}

\usepackage{subfigure}

\usepackage{ifthen} 

\provideboolean{shownotes} 
\setboolean{shownotes}{true} 
\newcommand{\margnote}[1]{
\ifthenelse{\boolean{shownotes}}%
{\marginpar{\raggedright\tiny\texttt{#1}}}%
{}%
}

\newcommand{\hole}[1]{
\ifthenelse{\boolean{shownotes}}%
{\begin{center} \fbox{ \rule {.25cm}{0cm}
\rule[-.1cm]{0cm}{.4cm} \parbox{.85\textwidth}{\begin{center}
\texttt{#1}\end{center}} \rule {.25cm}{0cm}}\end{center}}
{}
}

\theoremstyle{plain}

\newtheorem{lemma}{Lemma}[section]
\newtheorem{theorem}[lemma]{Theorem}
\newtheorem{proposition}[lemma]{Proposition}

\theoremstyle{definition}

\newtheorem{remark}[lemma]{Remark}
\newtheorem{definition}[lemma]{Definition}

\theoremstyle{remark}

\newcommand{\R}{\mathbb{R}}

\newcommand{\cF}{{\mathcal{F}}}

\newcommand{\cT}{{\mathcal{T}}}

\newcommand{\cL}{{\mathcal{L}}}

\newcommand{\cR}{{\mathcal{R}}}

\newcommand{\cG}{{\mathcal{G}}}

\newcommand{\bv}{\boldsymbol{v}}

\renewcommand{\Re}{\mathrm{Re}\,}

\newcommand{\bpz}{\bar{p}_0}
\newcommand{\blam}{\bar{\lambda}}

\newcommand{\bw}{\boldsymbol{w}}

\newcommand{\bD}{\mathbb{D}}

\newcommand{\ep}{\epsilon}

\begin{document}

\title[Derivation of a bacterial nutrient-taxis cross-diffusion system]{Derivation of a bacterial 
nutrient-taxis system with doubly degenerate cross-diffusion as the 
parabolic limit of a velocity-jump process}

\author[R.G. Plaza]{Ram\'on G. Plaza}

\address{Instituto de 
Investigaciones en Matem\'aticas Aplicadas y en Sistemas\\Universidad Nacional Aut\'onoma de 
M\'exico\\Circuito Escolar s/n, C.P. 04510 Cd. de M\'exico (Mexico)}

\email{plaza@mym.iimas.unam.mx}

\begin{abstract}
This paper is devoted to the justification of the macroscopic, mean-field nutrient taxis system with 
doubly degenerate cross-diffusion proposed by Leyva \emph{et al.} \cite{LMP1} to model the 
complex spatio-temporal dynamics exhibited by the bacterium \emph{B. subtilis} during experiments run 
\emph{in vitro}. This justification is based on a microscopic description of the movement of individual 
cells whose changes in velocity (in both speed and orientation) obey a \emph{velocity jump process} 
\cite{ODA88}, governed by a transport equation of Boltzmann type. For that purpose, the asymptotic method 
introduced by Hillen and Othmer \cite{HiOt00,OtHi02} is applied, which consists of the computation of the 
leading order term in a regular Hilbert expansion for the solution to the transport equation, under an 
appropriate parabolic scaling and a first order perturbation of the turning rate of Schnitzer type 
\cite{Schn93}. The resulting parabolic limit equation at leading order for the bacterial cell density 
recovers the degenerate nonlinear cross diffusion term and the associated chemotactic drift appearing in the 
original system of equations. Although the bacterium \emph{B. subtilis} is used as a prototype, the method 
and results apply in more generality.
\end{abstract}



\maketitle

\setcounter{tocdepth}{1}

\section{Introduction}

\label{intro}

When grown \emph{in vitro}, many species of bacteria exhibit a variety of complex spatial patterns which 
depend upon environmental conditions such as the level of nutrients or the hardness of the substrate. 
Continuous mean-field systems of partial differential equations for the bacterial density and the nutrient 
concentration have been proposed to simulate these patterns with moderate to relative great success (see, for 
example, Murray \cite{MurI3ed}, the discussion and criticism by Ben-Jacob and collaborators 
\cite{B-JLev06,GKCB}, and the references therein). In a recent contribution, Leyva \emph{et al.} \cite{LMP1} 
introduced a system of reaction-diffusion and nutrient taxis system of equations to model the complex 
spatio-temporal dynamics exhibited by strains of the bacterium \emph{Bacillus subtilis} observed in 
experiments (cf. \cite{OMM}). The model reads
\begin{equation}
\label{fullsyst}
\begin{aligned}
u_t &=  \nabla \cdot \big( \sigma uv \nabla u \big) - \nabla \cdot \left( \sigma u^2 v \frac{\chi_0 K_d}{(K_d 
+ v)^2)} \nabla v \right) +  
\theta k uv, \\
v_t &= D_v \Delta v - k uv, 
\end{aligned}
\qquad \quad x \in \Omega, \; \; t >0,
\end{equation}
for scalar unknowns $u = u(x,t)$ and $v = v(x,t)$, where $\Omega \subset \R^n$, $n = 1,2$, is a bounded open 
domain, with piecewise regular boundary, and $\sigma, \theta, k, K_d, D_v, \chi_0 > 0$ are positive 
constants. In addition, $u$ and $v$ satisfy no-flux boundary conditions of the form
\begin{equation}
\label{bcs}
\begin{aligned}
\left(uv \nabla u - u^2 v \frac{\chi_0 K_d}{(K_d + v)^2)} \nabla v \right) \cdot \hat{\nu} &= 0, \\
\nabla v \cdot \hat{\nu} &= 0,
\end{aligned}
\qquad \quad x \in \partial \Omega, \; \; t >0,
\end{equation}
where $\hat{\nu} \in \R^n$, $|\hat{\nu}| = 1$, denotes the outer unit normal at each point of $\partial 
\Omega$. System \eqref{fullsyst} is further endowed with initial conditions,
\begin{equation}
\label{incond}
u(x,0) = u_0(x), \; \;  \;\;v (x,0) = v_0(x), \qquad x\in \Omega,
\end{equation}
where $u_0$ and $v_0$ are known functions. In this model, $u$ and $v$ represent the bacterial density and the 
nutrient concentration, respectively. Thus, $u_0$ and $v_0$ represent the initial spatial distribution of 
bacteria and nutrient. The constant $\sigma$ is related to the hardness of the substrate on the Petri dish 
(usually agar or a similar substance), $\theta$ is the conversion rate factor in the signal consumption 
mechanism, $K_d$ is the receptor-ligand dissociation constant, $k$ is the intrinsic consumption rate, and 
$D_v$ is the diffusion coefficient for the nutrient. Notice that the equation for the bacterial density 
contains a chemotactic term that models the response of the bacteria towards changes in the gradient of a 
chemical signal (in this case, originated by the nutrient). The constant $\chi_0 > 0$ measures the intensity 
of such signal.

One of the main features of system \eqref{fullsyst} is the appearance of a non-linear, degenerate, 
cross-diffusion coefficient in the bacterial density equation of the form
\begin{equation}
\label{diffcoeff}
D_u(u,v) = \sigma uv,
\end{equation}
which was originally proposed by Kawasaki \emph{et al.} \cite{KMMUS}. When grown 
on semi-solid agar plates with poor-to-medium nutrient levels, the \emph{B. subtilis} strain form densely 
branched, highly structured stable patterns (cf. Oghiwari \emph{et al.} \cite{OMM}). Experiments suggest that 
the bacterial colony envelope is a boundary layer of cells whose activity is high around the front, but 
immotile either inside of the envelope (where the nutrient is almost completely consumed) or in the outermost 
region (where the bacterial density is still low). Kawasaki and collaborators modelled this behaviour by 
allowing the diffusivity of the bacteria $D_u$ to be low when either the nutrient concentration or the 
bacterial density are low, and consequently proposed the form of the diffusion coefficient in 
\eqref{diffcoeff}. The original model by Kawasaki \emph{et al.} is a reaction-diffusion system with no 
transport terms due to taxis. 

The motivation of Leyva and collaborators to extend Kawasaki's model originated from the theoretical and 
experimental work by E. Ben-Jacob and his group \cite{B-JCoLev00,CCB,GKCB}, who argued in favour of the 
chemotactic response of bacteria as an essential ingredient to understand some of the fundamental features of 
the spatio-temporal patterns observed in experiments. Chemotactic terms in Keller-Segel type models 
\cite{KeSe1,KeSe2} usually appear in the equation for bacterial density as the divergence of a 
\emph{chemotactic flux}, $\boldsymbol{J}_c$, having the generic form,
\[
\boldsymbol{J}_c = - \zeta(u,v) \chi(v) \nabla v.
\]
The function $\chi = \chi(v)$ is known as the \emph{chemotactic sensitivity function} and it usually depends 
on the chemical concentration alone. The function $\zeta =  \zeta(u,v)$ was labeled by Ben-Jacob as the 
\emph{bacterial response function}. It measures the response of the bacteria to the effective sensed 
gradient, namely, $\chi(v) \nabla v$. Based on experimental observations \cite{CCB}, Ben-Jacob and 
collaborators \cite{B-JCoLev00,GKCB} proposed an empirical rule: the bacterial response function should be 
proportional to the product of the bacterial density and its corresponding diffusion 
coefficient\footnote{notice that when the diffusion coefficient is constant, $D_u \equiv 1 > 0$ (after 
normalizations), this rule implies that the bacterial response function should be $\zeta = u$, and the 
chemotactic flux reads $\boldsymbol{J}_c = - u \chi(v) \nabla v$, recovering the standard Keller-Segel model 
\cite{KeSe1,KeSe2}.}, that is,
\begin{equation}
\label{empiricalrule}
|\zeta| \propto u D_u.
\end{equation}

Based on this suggestion and on the form of the effective diffusion coefficient \eqref{diffcoeff}, Leyva 
\emph{et al.} \cite{LMP1} proposed a chemotactic flux of the form
\[
\boldsymbol{J}_c = - \sigma u^2 v \chi(v) \nabla v,
\]
where the chemotactic sensitivity obeys a Lapidus-Schiller receptor law \cite{LapSch1},
\[
 \chi(v) = \frac{\chi_0 K_d}{(K_d + v)^2}.
\]
The result is the taxis term appearing in \eqref{fullsyst}. This new chemotactic term affects the velocity of 
the colony envelope as well as the morphology of the patterns, as numerical simulations and asymptotic 
calculations actually show (see \cite{LMP1}). In addition, the presence of an increasing chemotactic signal
supresses the onset for instabilities \cite{BMP17} (for a related discussion, see \cite{ArLe}). The new model 
seems to capture well the complexity of the observed patterns. In Figure \ref{figsimul}, the results of the 
numerical computation of a non-dimensionalization of system \eqref{fullsyst} - \eqref{incond} (see equations 
\eqref{onondimsyst} - \eqref{onondimic} below) are presented. The simulations were performed in the dense 
branch morphology regime, that is, for medium-to-soft agar and poor level of nutrients. The initial 
condition for the bacterial density is a Gaussian function resembling inoculation of the strain at the center 
of the Petri dish. After some time, the envelope front moves outward, forming spatial patterns which exhibit 
more activity near its boundary.
\begin{figure}[h]
\begin{center}
\subfigure[$t = 0.463 \sim 0^+$]{\label{figsimula}\includegraphics[scale=.25, clip=true]{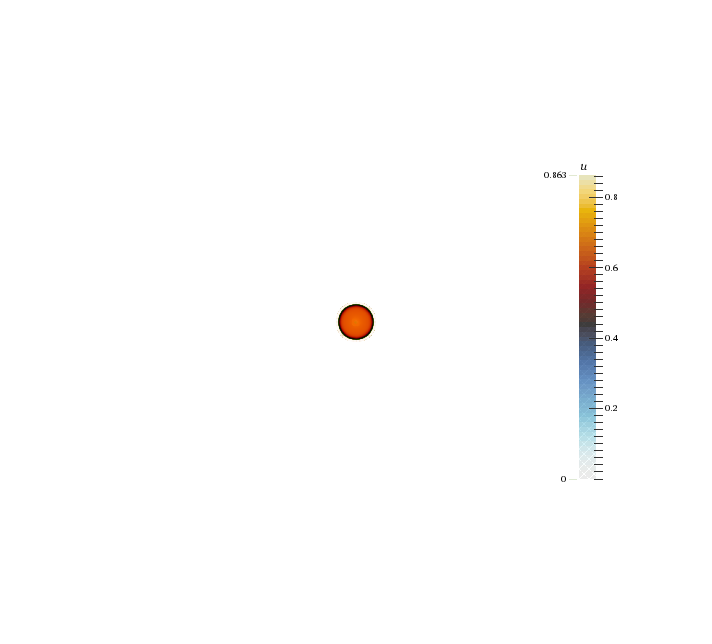}}
\subfigure[$t = 5.098$]{\label{figsimuld}\includegraphics[scale=.25, clip=true]{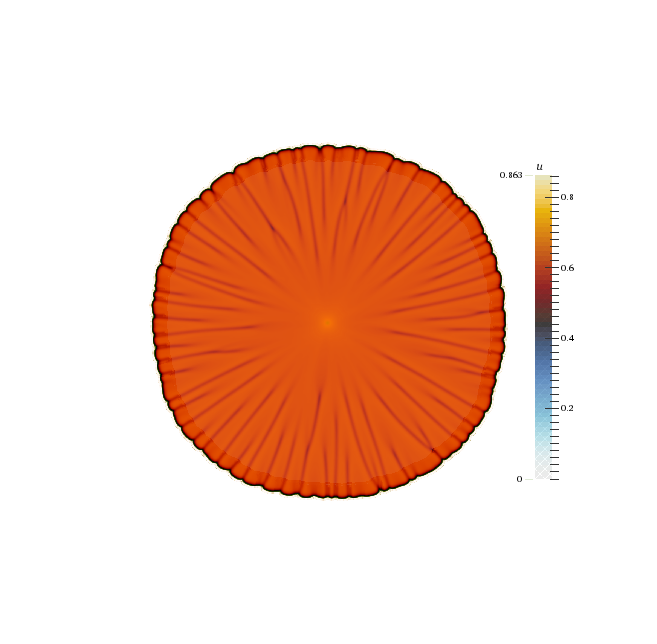}}
\end{center}
\caption{Numerical simulation of the non-dimensional system \eqref{onondimsyst} - 
\eqref{onondimic} on a square domain $\Omega = (-L/2,L/2) \times 
(-L/2, L/2)$, with $L = 680$. The pictures show the computed values of the bacterial density $u$. The initial 
conditions were $v_0(x,y) \equiv 0.71$ (uniform nutrient initial distribution) and $u_0(x,y) = 0.71 e^{-(x^2 
+ y^2)/6.25}$. The bacterial colony is depicted at a time very close to the initial distribution, $t \sim 
0^+$ (figure \ref{figsimula}), and for a later time step where the envelope front has already spreaded out 
(figure \ref{figsimuld}). The parameter values considered were 
$\sigma_0 = 4.0$ (soft agar) and $\chi_0 = 2.5$ (color plot online).}\label{figsimul}
\end{figure}

Density-dependent diffusion coefficients are used in many biological contexts, such as continuous models in 
population biology \cite{Aron80,GuMaC77,MyKr74}, spatial ecology \cite{SKT79}, and eukaryotic cell biology 
\cite{SePO07}, just to mention a few. From the mathematical viewpoint, when the nonlinearity of the diffusion 
coefficient is degenerate (meaning that diffusion approaches zero when the density does also), these 
equations or systems of equations are endowed with interesting properties. Among the new mathematical 
features one finds that equations with degenerate diffusion might exhibit finite speed of propagation 
of initial disturbances \cite{GiKe96}, in contrast with the strictly parabolic case. Another property is the 
emergence of traveling waves of ``sharp'' type (cf. \cite{SaMaKa96b,Sh10}). The cross-diffusion system 
\eqref{fullsyst} has been the subject of recent investigations. Among the new findings we mention the global 
existence of bounded weak solutions and the absence of finite time blow-up \cite{Win17pre}, even in the case 
of several space dimensions \cite{PlWi17a}. Thus, density-dependent diffusion systems with appropriately 
coupled taxis terms warrant attention from the community working on the analysis of partial differential 
equations thanks to their rich mathematical structure.

The underlying problem when using macroscopic field equations like \eqref{fullsyst} is how to justify the 
form of the chemotactic flux and/or of the diffusion coefficient which, under a pure phenomenological 
approach, are simply postulated based on empiric considerations. Another issue is how to incorporate 
microscopic responses of individual cells into the chemotactic sensitivity. There is, however, a stochastic 
process that has become standard (and appropriate) for describing the motion of cells, known as 
the \emph{velocity jump process}. Based on experimental observations \cite{Berg83,BeBr72,KeaLo03} showing 
that individual cells exhibit dynamics based on a run-and-tumble type of motion (see section \ref{secplvjp} 
below), kinetic transport equations of Boltzmann type with jumps in velocity space are common to model 
the movement of biological agents and their chemotactic responses (see \cite{Ptl1,Alt80,ODA88} and the 
references therein). Moreover, kinetic equations have been also used to derive macroscopic mean-field 
continuous equations after a limiting procedure (cf. \cite{HiOt00,OtHi02,ErOt04,HiPa13}), 
involving outer (or Hilbert) asymptotic expansions which are customary in kinetic theory 
\cite{Ells73,HabMat75,LaKe74}.

The purpose of this work is to take a closer look at the mathematical modelling approach that leads to the 
particular form of system \eqref{fullsyst}, by showing that the former can be retrieved as the parabolic 
diffusion limit of a velocity-jump process. Under the assumptions of a parabolic scaling together with an 
unbiased unperturbed turning kernel and a perturbed turning frequency (which depends on the chemical and 
population densities), it is shown that the leading order term in the Hilbert expansion to solutions of the 
Boltzmann type transport equation is a solution to the equation for the bacterial density in 
\eqref{fullsyst}. The key ingredient is to consider that the unperturbed turning frequency does not depend on 
the velocity, but does depend on the zero velocity moment (or marginal density) of the agent population as 
well as on the chemical concentration. Furthermore, the next order in the expansion of the turning 
frequency does depend on the velocity and encodes the bacterial response to the chemotactic signal. In 
addition, it is also shown that kinetic no-mass-flux boundary conditions for solutions to the transport 
equation lead to the no flux boundary conditions for $u$ in \eqref{bcs}. These observations are summarized in 
Theorem \ref{mainprop}. It is important to mention that this procedure is strictly formal. Yet, the 
main purpose of this derivation is to connect the forms of the chemotactic flux 
term and of the density dependent cross-diffusion coefficient \eqref{diffcoeff} with microscopic features of 
the underlying velocity-jump process such as the bias due to the chemotactic signal and the turning frequency 
of individuals. In this fashion, a simple microscopic interpretation of the (otherwise purely 
phenomenological) choices for the particular form of the diffusion and advection terms appearing in 
\eqref{fullsyst} is presented. Up to our knowledge, this work is the first attempt to derive degenerate 
cross-diffusion coefficients and chemotactic sensitivities which are density-dependent and related to each 
other, from a velocity jump process. Although the bacterium \emph{B. subtilis} and system \eqref{fullsyst} 
are used as a prototype, the methodology and results apply more generally.

The plan of the paper is as follows. In section \ref{secplvjp} the master transport equation 
governing velocity jump processes is recalled. Section \ref{secmeth} contains a description of the
method introduced by Hillen and Othmer \cite{HiOt00,OtHi02}, and the interior limit under a 
parabolic scaling is presented. Particular attention is devoted to boundary conditions. It is shown that, in 
the limit, the no-flux boundary conditions in \eqref{bcs} depend on the no normal 
mass flux across the boundary nature of the conditions imposed on the solutions to the transport equations, 
regardless of their particular form. The central section \ref{secmodelling} contains the mathematical 
modelling of the perturbed turning kernel and frequency that leads to the explicit form of system 
\eqref{fullsyst}. Finally, in section \ref{secdiscuss} the main features of this derivation and some 
potential consequences, both at the mathematical and at the biological levels, are discussed.

\section{Velocity-jump processes and formal asymptotic expansions}
\label{secplvjp}

Experimental observations show that some species of bacteria alternate two basic modes of motion: a simple 
linear motion with constant speed, called a \emph{run}, followed by a \emph{tumble}, which is basically a 
reorientation of the cell (for experimental evidence of this behaviour see \cite{BeBr72,Berg83} for the 
species \emph{E. coli}, and \cite{KeaLo03} for \emph{B. subtilis}). During a tumble the bacterium's location 
remains essentially constant while it spins around. After a tumble, the cell selects a new direction of 
movement. Experiments show that (i) the duration of a tumble  is short in comparison to that of a run; (ii) 
the distribution of the change in direction between the run preceding a tumble and the following one is the 
same from tumble to tumble; and, (iii) the speed of the bacterium during a run does not change from run to 
run. To sum up, the frequency of tumbles seems to be a function of time, position and direction, whereas run 
speed and distribution of the change in direction due to a tumble appear to be relatively constant. This 
run-and-tumble type of bacterial movement can be efficiently described by a Boltzmann-type transport equation 
from kinetic theory. This description is called a velocity-jump process, and it is based on the assumption 
that generic agents (which can be particles, cells or microorganisms) make instantaneous jumps in velocity 
space rather than in physical space. Thus, the motion consists on a sequence of runs separated by 
reorientations during which a new velocity is chosen. It is assumed that the changes in velocity are the 
result of a Poisson process of intensity $\lambda > 0$ (called the \textit{turning frequency or turning 
rate}), that is, the probability distribution specifying the time between turning events is exponential with 
mean $1/\lambda$ (the mean run duration). 

In such processes, agents follow a given velocity $\bv' \in V \subset \R^n$, $n =1,2,3$, for a finite time. 
Here $V$ is the set of allowed velocities (often regarded as a symmetric with respect to the origin, compact 
set). A new velocity is chosen according to a function $T : V \times V \to \R$, where $T(\bv,\bv')$ is the 
probability density of turning from velocity $\bv'$ to velocity $\bv$, given that a reorientation occurs. $T$ 
is called the \textit{turning kernel}. Let $p = p(x,t,\bv)$ be the probability density function describing 
the population of agents which, at position $x \in \Omega \subseteq \R^n$ and at time $t>0$, travel with 
velocity $\bv \in V$. Here $\Omega \subseteq \R^n$ is an open set. In a velocity jump process, the evolution 
of $p$ is governed by the following transport (or forward Kolmogorov) equation (cf. \cite{ODA88}),
\begin{equation}
\label{vjp}
 \frac{\partial}{\partial t} p(x,t,\bv) + \bv \cdot \nabla_x p(x,t,\bv) = - \lambda p(x,t,\bv) + \lambda 
\int_V 
T(\bv, \bv') p(x,t,\bv') \, d\bv' + \mathcal{G},
\end{equation}
where $\mathcal{G}$ accounts for the rate of change in $p$ due to reaction. The drift term $\bv \cdot 
\nabla_x p$ represents straight runs with velocity $\bv$, whereas the absorption term $- \lambda p$ on the 
right hand side corresponds to agents leaving the ``state'' $(x,\bv)$ (position $x$ with velocity $\bv$). The 
integral term accounts for the agents jumping into the state $(x,\bv)$ after a reorientation. Since $T : V 
\times V \to \R_+$ is a probability density, we assume it is non-negative and normalized so that
\[
\int_V T(\bv, \bv') \, d\bv' = 1,
\]
for all $\bv \in V$. For many purposes, one does not need to determine the distribution $p = p(x,t,\bv)$ but 
only some of its first velocity moments, such as the marginal density (or zero moment),
\[
\rho(x,t) = \int_V p(x,y,\bv) \, d\bv,
\]
which is the density of individuals at position $x$ at time $t > 0$, whatever their velocity.

Velocity-jump processes have been often used to model the movement of biological individuals, such as the 
bacteria \textit{E. coli} \cite{Berg83} and \textit{B. subtilis} \cite{MRFSASS17}, or even reef fish larvae 
\cite{CHPS04}. A backward equation similar to \eqref{vjp} was first derived by Stroock \cite{Stro74} to 
describe random motion of bacteria. We employ the transport equation \eqref{vjp} as the master equation to 
describe the motion of biological agents, as proposed by Othmer \textit{et al.} \cite{ODA88}. Our goal is to 
derive an approximate, continuous, mean-field limit system that encodes an external bias imposed on the 
motion due to a chemical signal. For that purpose we follow the general method by Hillen and Othmer 
\cite{HiOt00,OtHi02}, which is based on regular perturbation (Hilbert) expansions on solutions to equation 
\eqref{vjp}.

\subsection{The method of Hillen and Othmer}
\label{secmeth}

Let $\Omega \subseteq \R^n$ be an open domain. The set of admissible velocities $V \subset \R^n$ is assumed 
to 
be symmetric with respect to the origin (that is, $\bv \in V \Rightarrow -\bv \in V$) and compact. Let us 
assume that $\lambda > 0$ is constant in $\bv$. Let $\mathcal{K} := \{ f \in L^2(V) \, : \, f \geq 0\}$ be 
the cone of non-negative functions in $L^2(V)$. Then, for fixed $(x,t) \in \Omega \times (0,+\infty)$ we 
regard
\begin{equation}
\label{defcT}
\cT p(\bv) := \int_V T(\bv, \bv') p(x,t,\bv') \, d \bv',
\end{equation}
as an integral operator in $L^2(V)$. Hillen and Othmer \cite{HiOt00} impose the following structural 
assumptions on $T$.
\begin{equation}
\begin{minipage}[c]{4in}
$T(\bv,\bv') \geq 0$ for all $(\bv,\bv') \in V \times V$; $\int_V T(\bv,\bv') \, d\bv = 1$, for all $\bv' \in 
V$, and $\int_V \int_V T(\bv,\bv')^2 \, d\bv \, d \bv' < +\infty$.
\end{minipage}
\label{T1}
\tag{T$_1$}
\end{equation}

\medskip

\begin{equation}
\label{T2}
\tag{T$_2$}
\begin{minipage}[c]{4in}
There exist functions $\eta_0, \phi, \psi$ in $\mathcal{K}$ satisfying $\eta_0 \not\equiv 0$, $\phi, \psi 
\neq 
0$ a.e., such that
\[
\eta_0(\bv) \phi(\bv') \leq T(\bv,\bv') \leq \eta_0(\bv) \psi(\bv'),
\]
for all $(\bv,\bv') \in V \times V$.
\end{minipage}
\end{equation}

\medskip

\begin{equation}
\label{T3}
\tag{T$_3$}
\begin{minipage}[c]{4in}
$\| \mathcal{T} \|_{\langle 1 \rangle^\perp \to \langle 1 \rangle^\perp} < 1$, where $\langle 1 \rangle^\perp 
\subset L^2(V)$ is the orthogonal complement of constants in $L^2(V)$.
\end{minipage}
\end{equation}

\medskip

\begin{equation}
\label{T4}
\tag{T$_4$}
\begin{minipage}[c]{4in}
$\int_V T(\bv,\bv') \, d\bv' = 1$, for all $\bv \in V$.
\end{minipage}
\end{equation}

\medskip

For every given $\lambda > 0$, constant in $\bv \in V$, the \textit{turning operator}, $\cL : L^2(V) \to 
L^2(V)$, is defined as
\begin{equation}
\label{turningop}
\cL p(\bv) := -\lambda p(\bv) + \lambda \cT p(\bv) = -\lambda p(\bv) + \lambda \int_V T(\bv,\bv') p(\bv') \, 
d\bv'.
\end{equation}

The following theorem summarizes the most important properties of the turning operator that we shall need.

\begin{theorem}[Hillen and Othmer \cite{HiOt00,OtHi02}]\label{theoHO}
Under assumptions \eqref{T1} - \eqref{T2} the following properties hold:
\begin{itemize}
\item[(a)] $\mu = 0$ is a simple eigenvalue of $\cL : L^2(V) \to L^2(V)$ with eigenfunction $f(\bv) \equiv 1$.
\item[(b)] For all $g \in \langle 1 \rangle^\perp$,
\[
\langle g, \cL g \rangle_{L^2} = \int_V g \cL g \, d\bv \leq - \mu_2 \| g \|_{L^2(V)}^2,
\]
where $\mu_2 = \lambda (1- \| \mathcal{T} \|_{\langle 1 \rangle^\perp \to \langle 1 \rangle^\perp})$.
\item[(c)] All non-zero eigenvalues $\mu$ of $\cL$ satisfy the estimate $-2 \lambda < \Re \mu \leq - \mu_2 < 
0$, and to within scalar multiples there is no other positive eigenfunction.
\item[(d)] $\| \cL \|_{L^2 \to L^2} \leq 2 \lambda$.
\item[(e)] $\cL$ restricted to $\langle 1 \rangle^\perp \subset L^2(V)$ has a linear pseudo-inverse $\cF$ 
with 
norm
\[
\| \cF \|_{\langle 1 \rangle^\perp \to \langle 1 \rangle^\perp} \leq \frac{1}{\mu_2}.
\]
\end{itemize}
\end{theorem}

Let us now consider a velocity-jump processes (for simplicity in the course of this section, in the absence 
of reaction terms or external forces) governed by the equation
\begin{equation}
\label{wvjp}
 \frac{\partial}{\partial t} p(x,t,\bv) + \bv \cdot \nabla_x p(x,t,\bv) = - \lambda p(x,t,\bv) +  \int_V 
\lambda T(\bv, \bv') p(x,t,\bv') \, d\bv',
\end{equation}
where the turning rate $\lambda$ may as well depend on the velocity $\bv \in V$ and on other external 
signals, represented by $S = S(x,t)$ (for example, the density function of a chemo-attractant or a 
chemo-repellent substance). $\lambda$ may also depend on $(x,t)$ via other state variables such as the 
marginal density $\rho = \rho(x,t)$. As Hillen and Othmer point out, the dependence of $\lambda$ is not only 
on $S$ pointwise (on the value of $S$ at $(x,t)$), but also on $\nabla S$, or nonlocally on $S$ as well. 
Thus, following \cite{OtHi02}, we write $\lambda = \lambda(\bv,\hat S)$ to indicate this behaviour. 

According to custom in the study of transport equations, we consider a scaling limit that allows to 
approximate the description of the motion by a simpler model, such as a diffusion- or drift-dominated 
partial differential equation. Here we adopt a parabolic (or diffusion) scaling with its corresponding 
Hilbert expansion and limiting equation (cf. \cite{HiOt00,Hil04,HiPa13}). The parabolic scaling takes the form
\begin{equation}
\label{scal}
\tau = \ep^2 t, \quad \xi = \ep x,
\end{equation}
under the fundamental assumption that
\begin{equation}
\label{ascal}
\tau, \, \xi \, = O(1), \quad \text{as} \;\; \ep \to 0^+.
\end{equation}

\begin{remark}
In the description of the motion of bacteria, for instance, these quantities can be determined 
experimentally, leading to the justification of the use of this scaling in certain regimes of the physical 
parameters under consideration. For example, the bacterium {\em E. coli} on a Petri dish displays a 
characteristic speed of $s \approx 10$ $\mu$m/sec. and a mean run time of $1/\lambda \approx 1$ sec. (cf. 
\cite{OtXu13}). Under similar experimental conditions, strains of {\em B. subtilis} show characteristic 
speeds of $s \approx 1$ $\mu$m/sec. (cf. \cite{HPK95}) with average run time of $1 / \lambda \approx 1$ sec. 
(cf. \cite{ITFK05}). Choosing a typical experimental length scale of $L = 1$ mm. and a time scale of $T = 
10^4$ sec., we have that $s = O(\ep^{-1})$ and $\lambda = O(\ep^{-2})$, with $\epsilon = 10^{-2}$ for {\em E. 
coli} and $\ep = 10^{-1}$ for {\em B. subtilis}. Adopting the scaling \eqref{scal} we have that, in both 
cases, the assumption \eqref{ascal} holds for a small, non-dimensional parameter $\ep > 0$.
\end{remark}

Rewriting equation \eqref{wvjp} in the new space and time coordinates, $(\xi,\tau)$, we obtain
\begin{equation}
\label{svjp}
\ep^2 \frac{\partial}{\partial \tau} p(\xi,\tau,\bv) + \ep \bv \cdot \nabla_\xi p(\xi,\tau,\bv) = - \lambda 
p(\xi,\tau,\bv) + \int_V \lambda T(\bv,\bv') p(\xi,\tau,\bv') \, d\bv',
\end{equation}
where $(\xi,\tau) \in \widetilde{\Omega} \times (0,+\infty)$, $\widetilde{\Omega} := \ep \Omega = \{ \ep x \, 
: \, x \in \Omega\}$, $\bv \in V$.

We are interested in the case where the velocity-jump process is influenced by the presence of an external
chemical field, whose concentration in the scaled variables will be denoted as $S = S(\xi,\tau)$. Following 
\cite{OtHi02} we propose a first order perturbation of a certain turning frequency $\lambda_0$ 
\textit{independent of} $\bv$, together with an unperturbed turning kernel $T_0$ that depends only on 
$(\bv,\bv') \in V \times V$:
\begin{equation}
\label{ptr}
\begin{aligned}
T &= T_0(\bv,\bv'), \\
\lambda(\bv, \hat{S}) &= \lambda_0(\xi,\tau,\hat{S}) + \ep \lambda_1(\bv,\hat{S}).
\end{aligned}
\end{equation}
Here the turning kernel is unperturbed from a normalized turning kernel $T_0 : V \times V \to \R_+$ which is 
assumed to satisfy assumptions \eqref{T1} - \eqref{T4}, whereas $\lambda_0$ may depend on $(\xi,\tau)$ via 
the marginal density $\rho$ and the chemical concentration $S$. The important fact is that $\lambda_0$ does 
not depend on $\bv$ (it is constant with respect to the operator acting on $L^2(V)$ for each 
fixed $(\xi,\tau)$). The perturbation $\lambda_1$ is assumed to be bounded. Thus, we denote
\begin{equation}
\label{defL0}
\cL_0 p(\bv) := - \lambda_0 p(\bv) + \lambda_0 \int_V T_0(\bv,\bv') p(\bv') \, d\bv', \qquad \quad \cL_0 : 
L^2(V) \to L^2(V),
\end{equation}
an operator which belongs to the class of operators acting on $L^2(V)$ as described in Theorem \ref{theoHO}.

Under our parabolic scaling, it is assumed that the population density underlies a regular perturbation 
expansion of the form
\begin{equation}
\label{hilexp}
p(\xi,\tau,\bv) = p_0 (\xi,\tau,\bv) + \ep p_1 (\xi,\tau,\bv) + \ep^2 p_2(\xi,\tau,\bv) + O(\ep^3),
\end{equation}
also called an \textit{outer or Hilbert expansion} in the context of kinetic theory (cf. 
\cite{Ells73,HabMat75,LaKe74}). Since the marginal density is independent of $\ep >0$ we may assume that
\begin{equation}
\label{pjzero}
\int_V p_j(\xi,\tau,\bv) \, d\bv = 0,
\end{equation}
for all $j \geq 1$ and all $(\xi, \tau) \in \widetilde{\Omega} \times (0,+\infty)$, without loss of 
generality.

\subsection{Interior parabolic limit}
\label{secipl}

In this section we derive a mean field equation for the leading order term in the Hilbert expansion 
\eqref{hilexp} which has been reported in \cite{OtHi02}. We include the details of such derivation for 
completeness. Let us assume that $(\xi, \tau) \in \widetilde{\Omega} \times (0,+\infty)$ are fixed. Being 
$\widetilde{\Omega}$ an open set, this means that we do not consider boundary conditions for the moment. Upon 
substitution of  the expansions \eqref{hilexp} and \eqref{ptr} into equation \eqref{svjp} we obtain a 
hierarchy of equations by comparing terms of equal order in powers of $\ep$. First, at leading order $O(1)$ 
we obtain the equation
\begin{equation}
\label{orderzero}
0 = - \lambda_0 p_0(\xi,\tau,\bv) + \lambda_0 \int_V T_0(\bv,\bv') p_0(\xi,\tau,\bv') \, d\bv' = \cL_0 
p_0(\xi,\tau,\bv).
\end{equation}

In view of Theorem \ref{theoHO}, $p_0(\xi,\tau,\bv)$ belongs to the kernel of $\cL_0$ with dimension equal 
to one (the zero eigenvalue is simple) and generated by the eigenfunction $f(\bv) = 1$. Thus, we conclude 
that $p_0$ does not depend on $\bv$,
\[
p_0(\xi,\tau,\bv) = \bpz(\xi,\tau)
\]
($p_0$ is constant in $\bv$ for each fixed $(\xi,\tau)$). This implies, in turn, that
\begin{equation}
\label{goodp0}
\int_V \bv \cdot \nabla_\xi \bpz(\xi,\tau) \, d\bv = 0,
\end{equation}
because $V$ is a symmetric set.

At the next order, $O(\ep)$, the resulting equation is
\[
\begin{aligned}
\bv \cdot \nabla_\xi \bpz(\xi,\tau) &= -\lambda_0 p_1(\xi,\tau,\bv) - \lambda_1(\bv,\hat{S}) \bpz(\xi,\tau) + 
\lambda_0 \int_V T_0(\bv,\bv') p_1(\xi,\tau,\bv) \, d\bv' + \\ & \;\; + \int_V \lambda_1(\bv',\hat{S}) 
T_0(\bv,\bv') \bpz(\xi,\tau) \, d\bv',
\end{aligned}
\]
which can be recast as
\begin{equation}
\label{orderone}
\cL_0 p_1(\xi,\tau,\bv) = \bv \cdot \nabla_\xi \bpz(\xi,\tau) + \lambda_1(\bv,\hat{S}) \bpz(\xi,\tau) - 
\int_V 
\lambda_1(\bv',\hat{S}) T_0(\bv,\bv') \bpz(\xi,\tau) \, d\bv'.
\end{equation}

It is well-known (cf. \cite{Kat80}) that if $\cL$ is a closed, linear operator and $\mu$ is an eigenvalue 
with finite multiplicity, then its complex conjugate, $\overline{\mu}$, is an eigenvalue of the formal 
adjoint $\cL^*$ with the same algebraic and geometric multiplicities. Hence, in view of Theorem \ref{theoHO}, 
we observe that $\mu = 0$ is a simple eigenvalue of $\cL^*$ with a unique positive eigenfunction 
$g(\bv) \equiv 1 \in L^2(V)$. Therefore, the right hand side of \eqref{orderone} must satisfy a solvability 
condition, inasmuch as
\[
\langle 1, \cL_0 p_1(\xi,\tau,\bv) \rangle_{L^2(V)} = \langle \cL_0^* \, 1\, , p_1(\xi,\tau,\bv) 
\rangle_{L^2(V)} = 0,
\]
as $\cL_0$ is singular. In view of \eqref{goodp0}, the solvability condition reads
\[
0 = \int_V \lambda_1(\bv,\hat{S}) \bpz(\xi,\tau) \, d\bv - \int_V \int_V \lambda_1(\bv',\hat{S}) 
T_0(\bv,\bv') 
\bpz(\xi,\tau) d\bv \, d\bv' ,
\]
a relation which is trivially satisfied because of \eqref{T1}. By Theorem \ref{theoHO} we can define a 
pseudo-inverse of $\cL_0$ on $\langle 1 \rangle^\perp \subset L^2(V)$, denoted by $\cF_0 = ({\cL_0}_{|\langle 
1 \rangle^\perp})^{-1}$, and, consequently, the solution to \eqref{orderone} can be expressed as
\begin{equation}
\label{solp1}
p_1(\xi,\tau,\bv) = \cF_0 \Big( \bv \cdot \nabla_\xi \bpz(\xi,\tau) + \lambda_1 (\bv,\hat{S}) \bpz(\xi,\tau) 
- 
\int_V \lambda_1(\bv',\hat{S}) T_0(\bv,\bv') \bpz(\xi,\tau) \, d\bv' \Big).
\end{equation}

Let us denote the operator
\[
\cR_1 p(\bv) := \bv \cdot \nabla_\xi p(\bv) + \lambda_1 (\bv,\hat{S}) p(\bv) - \int_V \lambda_1(\bv',\hat{S}) 
T_0(\bv,\bv') p(\bv') \, d\bv',
\]
for $p(\bv) \in L^2(V)$. Under assumptions \eqref{T1} - \eqref{T4} and by boundedness of $\lambda_1$, this is 
a linear operator acting on $L^2(V)$. Whence, the solution $p_1$ can be written in simplified form as 
\begin{equation}
\label{simpp1}
p_1(\xi,\tau,\bv) = \cF_0 \big(\cR_1 (\bpz(\xi,\tau))\big).
\end{equation}

At order $O(\ep^2)$ the equation for $p$ is given by
\begin{equation}
\label{ordertwo}
\begin{aligned}
\frac{\partial}{\partial \tau} \bpz(\xi,\tau) + \bv \cdot \nabla_\xi p_1(\xi,\tau,\bv) &= - 
\lambda_1(\bv,\hat{s}) p_1(\xi,\tau,\bv) - \lambda_0 p_2(\xi,\tau,\bv) + \\ &\;\; + \lambda_0 \int_V 
T_0(\bv,\bv') p_2(\xi,\tau,\bv') \, d\bv' + \\ &\;\; + \int_V \lambda_1(\bv',\hat{S}) T_0(\bv,\bv') 
p_1(\xi,\tau,\bv') \, d\bv',
\end{aligned}
\end{equation}
yielding,
\[
\begin{aligned}
\cL_0 p_2(\xi,\tau,\bv) &= \frac{\partial}{\partial \tau} \bpz(\xi,\tau) + \bv \cdot \nabla_\xi 
p_1(\xi,\tau,\bv) + \lambda_1(\bv,\hat{s}) p_1(\xi,\tau,\bv) + \\ & \;\; - \int_V \lambda_1(\bv',\hat{S}) 
T_0(\bv,\bv') p_1(\xi,\tau,\bv') \, d\bv'\\&= \frac{\partial}{\partial \tau} \bpz(\xi,\tau) + \cR_1 \big( 
p_1(\xi,\tau,\bv) \big).
\end{aligned}
\]
By the same argument, the solvability condition for $p_2$ can be obtained by taking the $L^2(V)$ product of 
last equation with the (constant) sole positive eigenfunction of $\cL_0^*$. The result is
\begin{equation}
\label{solvcond2}
0 = \int_V \Big[ \frac{\partial}{\partial \tau} \bpz(\xi,\tau) + \cR_1 \big( \cF_0 \big( 
\cR_1(\bpz(\xi,\tau)) 
\big) \big) \Big] \, d\bv,
\end{equation} 
after substituting \eqref{simpp1}. Let us now define 
\[
\blam_1(\bv,\hat{S}) := \int_V \lambda_1(\bv',\hat{S}) T_0(\bv,\bv') \, d\bv',
\]
which can be interpreted as the average bias, over all incoming velocities, of the turning rate to $\bv$. 
Thus,
\begin{equation}
\label{simpp11}
\begin{aligned}
p_1(\xi,\tau,\bv) &= \cF_0 \big( \cR_1 (\bpz(\xi,\tau)) \big) \\&= \cF_0 \Big( \bv \cdot \nabla_\xi 
\bpz(\xi,\tau) + \lambda_1 (\bv,\hat{S}) \bpz(\xi,\tau) - \int_V \lambda_1(\bv',\hat{S}) T_0(\bv,\bv') 
\bpz(\xi,\tau) \, d\bv'\Big) \\
&= \cF_0 \big( \bv \cdot \nabla_\xi \bpz(\xi,\tau) \big) + \cF_0 \big( (\lambda_1(\bv,\hat{S}) - 
\blam_1(\bv,\hat{S})) \bpz(\xi,\tau) \big).
\end{aligned}
\end{equation}

First note that, clearly,
\[
\int_V  \frac{\partial}{\partial \tau} \bpz(\xi,\tau) \, d\bv = |V| \frac{\partial}{\partial \tau} 
\bpz(\xi,\tau).
\]
Also, substituting \eqref{simpp11} we obtain
\begin{equation}
\label{ddot}
\begin{aligned}
\int_V \bv \cdot \nabla_\xi p_1(\xi,\tau,\bv) \, d\bv &= \int_V \bv \cdot \nabla_\xi \Big( \cF_0 \big( \bv 
\cdot \nabla_\xi \bpz(\xi,\tau) \big) \Big) \, d\bv + \\ \; &+ \int_V \bv \cdot \nabla_\xi \Big( \cF_0 \big( 
(\lambda_1(\bv,\hat{S}) - \blam_1(\bv,\hat{S})) \bpz(\xi,\tau) \big) \Big) \, d\bv.
\end{aligned}
\end{equation}
From the identity
\[
\bv \cdot \nabla_\xi \Big( \cF_0 \big( \bv \cdot \nabla_\xi \bpz \big) \Big) = \nabla_{\xi} \cdot \Big( [\bv 
(\cF_0 \bv)^\top] \nabla_\xi \bpz \Big)
\]
(observe that $\bv (\cF_0 \bv)^\top = \bv \otimes (\cF_0 \bv)$ is a tensor product), we can write the first 
integral on the right hand side of \eqref{ddot} as
\[
\int_V \bv \cdot \nabla_\xi \Big( \cF_0 \big( \bv \cdot \nabla_\xi \bpz(\xi,\tau) \big) \Big) \, d\bv = - 
\nabla_{\xi} \cdot \big( |V| \bD \nabla_\xi \bpz(\xi,\tau) \big),
\]
where the \textit{diffusion tensor} $\bD$ is defined as
\begin{equation}
\label{defD}
\bD := - \frac{1}{|V|} \int_V  \bv \otimes (\cF_0 \bv)   \, d\bv. 
\end{equation} 

Let us as well define the \textit{chemotactic velocity} $\bw_c$ as
\begin{equation}
\label{chemvel}
\bw_c := - \frac{1}{|V|} \int_V \bv \cF_0 \big(\blam_1(\bv,\hat{S}) -  \lambda_1(\bv,\hat{S}) \big) \, d\bv.
\end{equation}
Hence, from the identity 
\[
\int_V \bv \cdot \nabla_\xi \big( \beta \bpz \big) \, d\bv = \nabla_{\xi} \cdot \Big( \bpz \int_V \beta \bv  
\, 
d\bv \Big),
\]
which is valid for any scalar function $\beta = \beta(\xi,\tau,\bv,\hat{S})$, we recognize that the second 
integral on the right hand side of \eqref{ddot} can be written as
\[
\int_V \bv \cdot \nabla_\xi \Big( \cF_0 \big( (\lambda_1(\bv,\hat{S}) - \blam_1(\bv,\hat{S})) \bpz(\xi,\tau) 
\big) \Big) \, d\bv = \nabla_{\xi} \cdot \big( |V| \bpz(\xi,\tau) \bw_c \big).
\]
Therefore, the solvability condition \eqref{solvcond2} can be recast as
\[
\begin{aligned}
0 &= |V|  \frac{\partial}{\partial \tau} \bpz(\xi,\tau)  - \nabla_{\xi} \cdot \big( |V| \bD \nabla_\xi 
\bpz(\xi,\tau) \big) + \nabla_{\xi} \cdot \big( |V| \bpz(\xi,\tau) \bw_c \big) + \\ &+ \!\int_V 
\lambda_1(\bv,\hat{S}) \cF_0 \Big( \bv \cdot \nabla_\xi \bpz(\xi,\tau) + (\lambda_1(\bv,\hat{S}) - 
\blam_1(\bv,\hat{S}))\bpz(\xi,\tau)\Big) \, d\bv +\\ &- \!\int_V \int_V \!\lambda_1(\bv',\hat{S}) 
T_0(\bv,\bv')   \cF_0 \Big( \bv' \!\cdot \! \nabla_\xi \bpz(\xi,\tau) + (\lambda_1(\bv',\hat{S})\! - \!
\blam_1(\bv',\hat{S}))\bpz(\xi,\tau)\Big) d\bv' d\bv.
\end{aligned}
\]

The last two integrals cancel each other because of \eqref{T1} and Fubini's theorem. Thus, the resulting 
equation for $\bpz = 
\bpz(\xi,\tau)$ is
\begin{equation}
\label{parabeq}
\frac{\partial \bpz}{\partial \tau} = \nabla_{\xi} \cdot \big( \bD \, \nabla_\xi \bpz \big) - \nabla_{\xi} 
\cdot 
\big( \bpz \bw_c \big). 
\end{equation}
This is essentially a diffusion equation, with drift term given by the chemotactic velocity $\bw_c$, for the 
total marginal density 
\[
\rho(\xi,\tau) = |V| \bpz(\xi,\tau) = \int_V p(\xi,\tau,\bv) \, d\bv,
\]
in view of \eqref{pjzero}.

To sum up, we can state the following

\begin{proposition}[interior parabolic limit with perturbed turning rate]\label{propint}
Let $\Omega \subset \R^n$ be an open domain, and let $V \subset \R^n$ be a compact, symmetric with respect to 
the origin set of admissible velocities. Let $T_0 : V \times V \to \R$ be a turning kernel satisfying 
assumptions \eqref{T1} - \eqref{T4}. Furthermore, let us assume that there exists a small parameter $\ep > 0$ 
such that $\tau = \ep^2 t = O(1)$ and $\xi = \ep x = O(1)$, for all $(x,t) \in \Omega \times (0,+\infty)$ 
(parabolic scaling). If $p = p(\xi,\tau,\bv)$ is a solution to the scaled transport equation \eqref{svjp}, 
where the turning kernel admits a first order perturbation of the form $\lambda = \lambda_0 + \ep 
\lambda_1(\bv)$ (with $\lambda_0$ independent of $\bv$), then the leading order term $p_0$ of a regular 
Hilbert expansion $p = p_0 + \ep p_1 + \ep^2 p_2 + O(\ep^3)$ satisfies 
\[
p_0(\xi,\tau,\bv) = \bpz(\xi,\tau), \qquad \rho(\xi,\tau) = \int_V p(\xi,\tau,\bv) \, d\bv = |V| 
\bpz(\xi,\tau) ,
\]
where the marginal density $\rho = \rho(\xi,\tau)$  is a solution to the parabolic limit equation
\[
\frac{\partial \rho}{\partial \tau} = \nabla_{\xi} \cdot \big( \bD \, \nabla_\xi \rho  -  \rho 
\bw_c \big),
\]
for $\xi \in \widetilde{\Omega} = \epsilon \Omega$ (away from 
the boundary), and $\tau > 0$. The diffusion tensor $\bD \in \R^{n \times n}$ and the chemotactic velocity 
$\bw_c \in \R^n$ are given by
\[
\bD = - \frac{1}{|V|} \int_V \bv \otimes (\cF_0 \bv) \, d\bv,
\]
and
\[
\bw_c = - \frac{1}{|V|} \int_V \bv \cF_0 \big( \blam_1(\bv) - \lambda_1(\bv) \big) \, d\bv,
\]
respectively, where
\[
\blam_1(\bv) = \int_V \lambda_1(\bv') T_0(\bv,\bv') \, d\bv',
\]
is the average bias. Here $\cF_0$ denotes the pseudo-inverse of the operator  $\cL_0$ defined by 
\eqref{defL0}, restricted to the subspace $\langle 1 \rangle^\perp \subset L^2(V)$.
\end{proposition}

\subsection{Boundary conditions}

In this section, boundary conditions for the leading term $\bpz = \bpz(\xi,\tau)$ of the Hilbert 
expansion \eqref{hilexp} are derived, under the assumptions of a parabolic scaling with perturbed turning 
rate of the previous section. In general, these boundary conditions depend upon the boundary conditions 
imposed on the solutions to the transport equation. We shall see, however, that when the turning kernel is 
unperturbed and the turning frequency is $O(\epsilon)$-perturbed, the boundary conditions on the limit 
equation depend only on the no normal mass flux nature of the boundary conditions for the transport equation, 
independently of their form.

Let us assume that $\Omega \subset \R^n$ is a bounded domain with piecewise smooth boundary. In the rescaled 
variables, for each $\xi \in \partial \widetilde{\Omega}$ with $\widetilde{\Omega} = \epsilon \Omega$, we 
denote the outer unit normal vector at $\xi \in \partial \widetilde{\Omega}$ by $\hat{\nu}(\xi)$. The 
boundary of the phase space can be written as
\[
\partial \widetilde{\Omega} \times V = \Gamma_+ \cup \Gamma_- \cup \Gamma_0,
\]
where
\[
\Gamma_\pm := \{ (\xi,\bv) \in \partial \widetilde{\Omega} \times V \, : \, \pm \,  \bv \cdot \hat{\nu}(\xi) 
> 
0 \}, \quad \Gamma_0 := \{ (\xi,\bv) \in \partial \widetilde{\Omega} \times V \, : \, \bv \cdot 
\hat{\nu}(\xi) 
= 0 \}.
\]
We assume that $\Gamma_0$ is of zero measure with respect to $d \gamma_\xi d\bv$, where $d \gamma_\xi$ is the 
Lebesgue measure on $\partial \widetilde{\Omega}$, and define the trace spaces
\[
L^2_\pm := L^2 ( \Gamma_\pm ; |\bv \cdot \hat{\nu}(\xi)| d \gamma_\xi d\bv \}.
\]
Let us suppose that $p$ is regular enough, say, $p \in H^1(\widetilde{\Omega} \times V)$, so that the 
range of the trace operator lies in $L^2(\partial \widetilde{\Omega} \times V)$. Thus, let us 
denote $p_{|\Gamma_\pm} \in L^2_\pm$ as the trace of $p \in H^1(\widetilde{\Omega} \times V)$ on 
$\Gamma_\pm$, for fixed $\tau > 0$. For instance, if $p$ is smooth enough (at least of class $C^1$) then it 
is known that the trace of $p$ coincides with the continuous limit to the boundary, so 
that,
\begin{equation}
\label{tracecond}
p_{|\partial \widetilde{\Omega} \times V}(\xi,\tau,\bv) = \lim_{\substack{\widetilde{\xi} \in 
\widetilde{\Omega}\\ \widetilde{\xi} \to \xi}} \, p(\widetilde{\xi}, \tau, \bv), \qquad \text{for each } \, 
\xi \in \partial \widetilde{\Omega}.
\end{equation}
Likewise, if we denote, for each $\xi \in \partial \widetilde{\Omega}$,
\[
V^\pm := \{ \bv \in V \, : \, \pm \bv \cdot \hat{\nu}(\xi) > 0\},
\]
then we may as well define
\[
p_{|\partial \widetilde{\Omega} \times V}(\xi,\tau,\bv) =: \begin{cases}
p_{| {\Gamma_+}}(\xi, \tau,\bv), & \text{if} \, \bv \in V^+,\\
p_{| {\Gamma_-}}(\xi, \tau,\bv), & \text{if} \, \bv \in V^-.
\end{cases}
\]

In general, boundary conditions for solutions to transport equations have the form
\[
p_{|\Gamma_-} (\xi,\tau,\bv) = (\mathcal{B}p_{| {\Gamma_+}})(\xi, \bv,\tau), \quad (\xi,\bv) \in \Gamma_-, \, 
\tau > 0,
\]
expressing that the incoming flux of cells, $p_{|\Gamma_-}$, is related to the outgoing one, $p_{|\Gamma_+}$, 
through a linear bounded operator $\mathcal{B} : L^2_+ \to L^2_-$. 

Although making precise the form of the boundary operator $\mathcal{B}$ is necessary to solve the transport 
equation, for our purposes we only require that it satisfies the following \emph{no normal mass flux 
across the boundary} condition (see \cite{LeMi08}, section 2.4),
\begin{equation}
\label{nonormalflux}
\int_V p(\xi,\tau,\bv) \, (\bv \cdot \hat{\nu}(\xi)) \, d\bv = 0, \quad \text{for all } \; \xi \in \partial 
\widetilde{\Omega}, \; \tau > 0.
\end{equation}
Here $p = p(\xi,\tau,\bv) = p_{| {\Gamma_+}}(\xi, 
\tau,\bv) + (\mathcal{B}p_{| {\Gamma_+}})(\xi, \tau,\bv)$ denotes the trace of the solution to the transport 
equation $p = p(\xi,\tau,\bv)$ on $\partial \widetilde{\Omega} \times V$ for each fixed $\tau > 0$. Condition 
\eqref{nonormalflux} simply expresses that no agents (cells or particles) move across the boundary. 

An important class of no normal mass flux kinetic boundary conditions are the regular reflection boundary 
operators defined by Palczewski \cite{Plcz92} (see also \cite{Lods05}).

\begin{definition}[Palczewski \cite{Plcz92}]\label{defreflbc}
$\mathcal{B} \in \mathscr{L}(L_+^2, L_-^2)$ is a \emph{regular reflection boundary operator} if there exists 
a 
$C^1$-piecewise mapping $\mathcal{V} : \Gamma_- \to \R^n$ satisfying:
\begin{itemize}
\item[(a)] If $(\xi,\bv) \in \Gamma_-$ then $(\xi, \mathcal{V}(\xi,\bv)) \in \Gamma_+$.
\item[(b)] $(\mathcal{B}p)(\xi, \bv) = p(\xi, \mathcal{V}(\xi,\bv))$, for all $(\xi,\bv) \in \Gamma_-$, $p 
\in 
L^2_+$.
\item[(c)] $| \mathcal{V}(\xi,\bv)| = |\bv|$ for any $(\xi,\bv) \in \Gamma_-$.
\item[(d)] $|\hat{\nu}(\xi) \cdot \bv| = |\hat{\nu}(\xi) \cdot \mathcal{V}(\xi,\bv)| |\det \partial 
\mathcal{V}/\partial \bv|$, for all $(\xi,\bv) \in \Gamma_-$.
\item[(e)] $\mathcal{V}(\xi,\beta \bv) = \beta \mathcal{V}(\xi,\bv)$ for any $\beta > 0$, $(\xi,\bv) \in 
\Gamma_-$.
\end{itemize}
\end{definition}
\begin{remark}
As customary examples of regular reflection boundary operators in kinetic theory we have the 
\emph{bounce-back 
reflection condition},
\[
\mathcal{V}(\xi,\bv) = - \bv, \qquad (\xi,\bv) \in \Gamma_-,
\]
which states that an agent (particle or cell) hits the boundary and bounces back with the same velocity but 
with opposite direction, and the \emph{specular reflection boundary condition},
\[
\mathcal{V}(\xi,\bv) = \bv - 2(\bv\cdot \hat{\nu}(\xi)) \, \hat{\nu}(\xi), \qquad (\xi,\bv) \in \Gamma_-,
\]
which expresses that the boundary acts like a mirror and the agent is reflected making the same angle with 
respect to the tangent to the boundary. In both cases the speed $|\bv|$ is preserved (see Figure \ref{figbc}).
\end{remark}
\begin{figure}[h]
\begin{center}
\subfigure[Specular reflection]{\label{figbca}\includegraphics[scale=.175, clip=true]{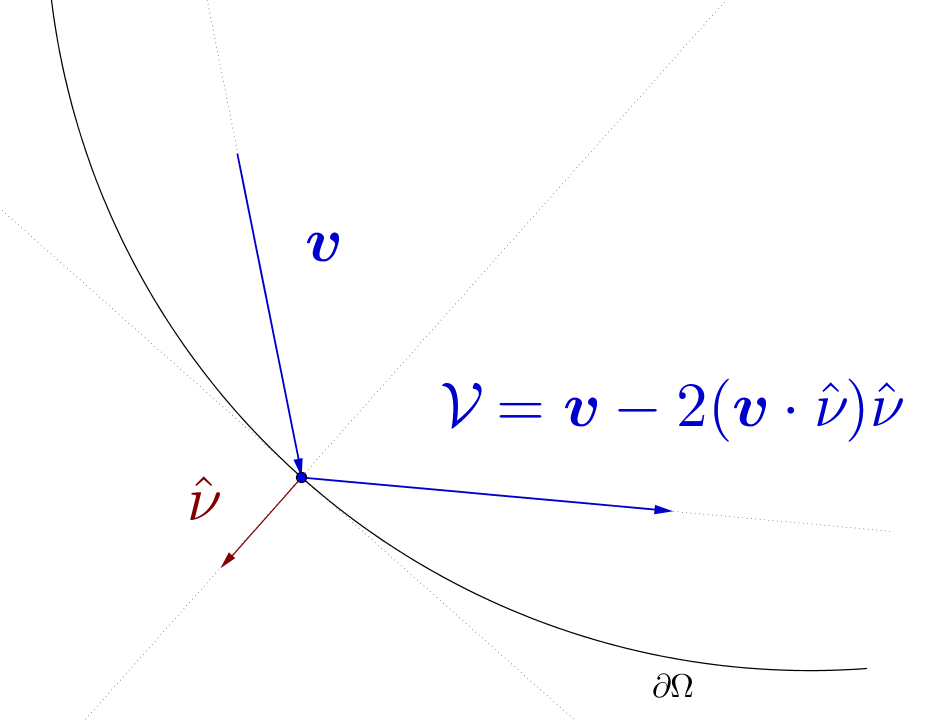}}
\subfigure[Bounce-back reflection]{\label{figbcb}\includegraphics[scale=.175, clip=true]{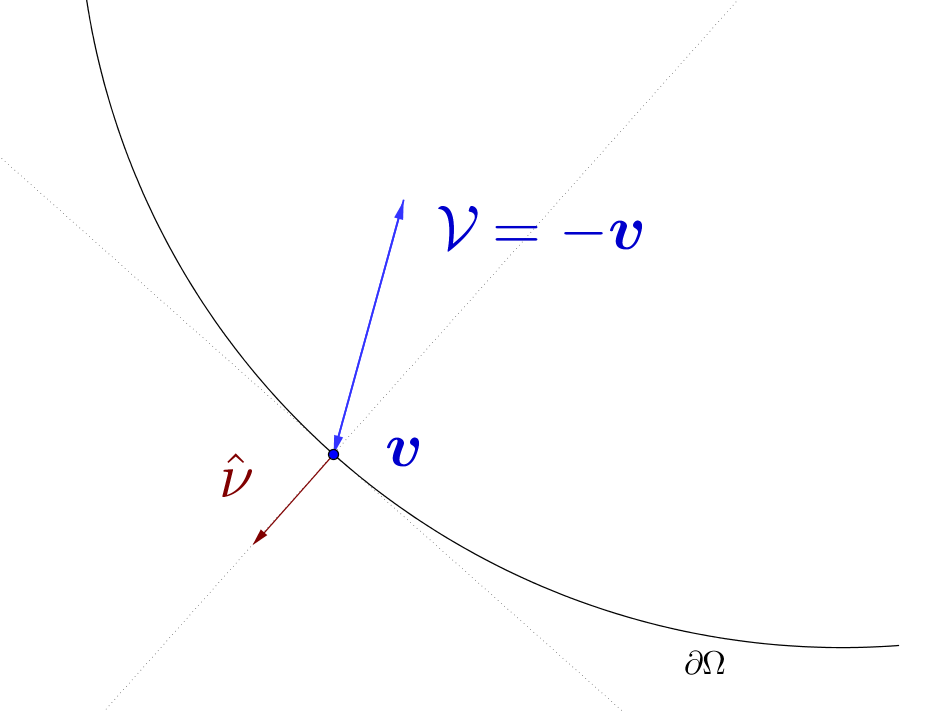}}
\end{center}
\caption{Typical examples of regular reflection boundary operators. In the specular reflection case (figure 
\ref{figbca}), the agent is reflected making the same angle with respect to the tangent to the boundary; in 
the case of the bounce-back operator (figure \ref{figbcb}) the agent is reflected exactly in the opposite 
direction. In both cases the speed is preserved (color plot online).}\label{figbc}
\end{figure}
\begin{lemma}
\label{lemreflec}
Every regular reflection boundary operator satisfies the no normal mass flux across the boundary condition 
\eqref{nonormalflux}.
\end{lemma}
\begin{proof}
 Follows directly from Definition \ref{defreflbc}; indeed, use (a), (b) and (d) to obtain
\[
\begin{aligned}
\int_V &p(\xi,\bv,\tau)(\bv \cdot \hat{\nu}(\xi)) \, d\bv = \!\!\int_{V^+} \!p_{|\Gamma_+} (\xi,\tau,\bv) 
\, 
(\bv \cdot \hat{\nu}(\xi)) \, d\bv + \!\!\int_{V^-} \!p_{|\Gamma_-} (\xi,\tau,\bv) \, (\bv \cdot 
\hat{\nu}(\xi)) \, 
d\bv \\
&= \int_{V^+} p_{|\Gamma_+} (\xi,\tau,\bv) \, (\bv \cdot \hat{\nu}(\xi)) \, d\bv - \int_{V^-} (\mathcal{B} 
p_{|\Gamma_+})(\xi,\tau,\bv) \, |\bv \cdot \hat{\nu}(\xi)| \, d\bv \\
&= \int_{V^+} \!p_{|\Gamma_+} (\xi,\tau,\bv)(\bv \cdot \hat{\nu}(\xi)) \, d\bv - \!\! \int_{V^-}\! p(\xi, 
\tau, 
\mathcal{V}(\xi,\bv)) \, |\hat{\nu}(\xi) \cdot \mathcal{V}(\xi,\bv)| |\det \partial \mathcal{V}/\partial \bv| 
\, d\bv \\
&= \int_{V^+}\! p_{|\Gamma_+} (\xi,\tau,\bv) \, (\bv \cdot \hat{\nu}(\xi)) \, d\bv - \int_{V^+} \!
p_{|\Gamma_+}(\xi, \tau, \bw) (\bw \cdot \hat{\nu}(\xi)) \, d \bw   \\
&= 0,
\end{aligned}
\]
for each $\xi \in \partial \widetilde{\Omega}$, $\tau > 0$, as claimed.
\end{proof}

\begin{remark}
The no normal mass flux across the boundary condition \eqref{nonormalflux} might as well be 
satisfied by other boundary conditions that could include, for instance, boundary operators of Maxwell type, 
in which reflection and diffusion effects are combined, even containing nonlocal terms (see 
\cite{BePr87,Lods05} for further information).
\end{remark}

In the sequel, we assume that the solution $p$ to the transport equation is smooth enough so that the trace 
can be computed via an interior limit (equation \eqref{tracecond}) and that the general kinetic boundary 
condition satisfies \eqref{nonormalflux}, regardless of its form. Since we are assuming that a regular 
Hilbert expansion is valid in the open domain $\widetilde{\Omega}$, we can compute the trace taking the limit 
to the boundary on  \eqref{hilexp}. In that case, the no normal mass flux across the boundary condition 
\eqref{nonormalflux} reads
\[
\int_V \big( p_0(\xi,\tau,\bv) + \epsilon p_1(\xi,\tau,\bv) + O(\epsilon^2) \big) (\bv \cdot \hat{\nu}(\xi)) 
\, d\bv = 0, \qquad \xi \in \partial \widetilde{\Omega}, \, \tau > 0.
\]
In view that this condition is independent of $\epsilon > 0$, we may assume that
\[
\int_V p_j(\xi,\tau,\bv) (\bv \cdot \hat{\nu}(\xi)) \, d\bv = 0,
\]
for all $j \geq 0$, $\xi \in \partial \widetilde{\Omega}$, $\tau > 0$. The condition for $j=0$ is trivially 
satisfied,
\[
\int_V p_0(\xi,\tau,\bv) (\bv \cdot \hat{\nu}(\xi)) \, d\bv = \bpz(\xi,\tau) \Big( \int_V \bv \, d\bv \Big) 
\cdot \hat{\nu}(\xi) = 0,
\]
inasmuch as $p_0$ is independent of $\bv$ and $V$ is symmetric. The condition for $j = 1$ will provide the 
boundary conditions for the leading term $\bpz = \bpz(\xi,\tau)$. The latter reads
\[
\int_V p_1(\xi,\tau,\bv) (\bv \cdot \hat{\nu}(\xi)) \, d\bv = 0.
\]
Substituting \eqref{simpp11} one arrives at
\begin{equation}
\label{approxbe}
\int_V  \Big( \cF_0 \big( \bv \cdot \nabla_\xi \bpz(\xi,\tau) \big) + \cF_0 \big( (\lambda_1(\bv,\hat{S}) - 
\blam_1(\bv,\hat{S})) \bpz(\xi,\tau) \big) \Big) (\bv \cdot \hat{\nu}(\xi)) \, d\bv = 0.
\end{equation}
Using the identity,
\[
\cF_0\big( \bv \cdot \nabla_\xi \bpz \big) (\bv \cdot \hat{\nu}(\xi)) = \Big( \bv (\cF_0 \bv)^\top \nabla_\xi 
\bpz \Big) \cdot \hat{\nu}(\xi)
\]
(which can be easily verified), we find that
\[
\begin{aligned}
\int_V \cF_0 \big( \big( \bv \cdot \nabla_\xi \bpz(\xi,\tau) \big) \big) (\bv \cdot \hat{\nu}(\xi)) \, d\bv 
&= 
\int_V \Big( \bv (\cF_0 \bv)^\top \nabla_\xi \bpz(\xi,\tau)  \Big) \cdot \hat{\nu}(\xi) \, d\bv \\&= \left[ 
\Big( \int_V \bv (\cF_0 \bv)^\top \, d\bv \Big) \nabla_\xi \bpz(\xi,\tau) \right] \cdot \hat{\nu}(\xi) \\
&= - |V| \big( \bD \nabla_\xi \bpz (\xi,\tau) \big) \cdot \hat{\nu}(\xi).
\end{aligned}
\]

On the other hand, we observe, since $\bpz$ is independent of $\bv$, that
\[
\begin{aligned}
\int_V  \cF_0 \big( (\lambda_1(\bv,\hat{S}) - \blam_1(\bv,\hat{S}))& \bpz(\xi,\tau) \big) (\bv \cdot 
\hat{\nu}(\xi)) \, d\bv = \\ &=\Big( \int_V \cF_0 (\lambda_1(\bv,\hat{S}) - \blam_1(\bv,\hat{S})) \, \bv \, 
d\bv 
\Big) \cdot \hat{\nu}(\xi) \bpz(\xi,\tau) \\
&= |V| (\bw_c \cdot \hat{\nu}(\xi)) \bpz(\xi,\tau).
\end{aligned}
\]
Therefore the boundary condition \eqref{approxbe} can be written as
\[
|V| \Big( \bD \nabla_\xi \bpz(\xi,\tau) - \bpz(\xi,\tau) \bw_c \Big) \cdot \hat{\nu}(\xi) = 0,
\]
or, equivalently,
\begin{equation}
\label{bcondi}
\Big( \bD \nabla_\xi \rho(\xi,\tau) - \rho(\xi,\tau) \bw_c \Big) \cdot \hat{\nu}(\xi) = 0, 
\end{equation}
for each $\xi \in \partial \widetilde{\Omega}$, $\tau > 0$.

We summarize these observations as follows.
\begin{proposition}[no flux boundary conditions in the parabolic limit]\label{propbc}
Under the assumptions of Proposition \ref{propint}, let us suppose that $\partial \Omega$ is piecewise 
smooth, that the regular Hilbert expansion \eqref{hilexp} of the solution to the transport equation 
\eqref{svjp} is smooth enough, and that the kinetic boundary conditions imposed on $p$ satisfy the no normal 
mass flux across the boundary condition \eqref{nonormalflux}. Then the leading term $\rho(\xi,\tau) = |V| 
\bpz(\xi,\tau)$ of the Hilbert expansion satisfies the following no flux boundary condition at $\partial 
\widetilde{\Omega}$,
\[
\Big( \bD \nabla_\xi \rho(\xi,\tau) - \rho(\xi,\tau) \bw_c \Big) \cdot \hat{\nu}(\xi) = 0, \qquad \xi \in 
\partial \widetilde{\Omega}, \; \tau > 0,
\]
whereupon the diffusion tensor $\bD$ and the chemotactic velocity $\bw_c$ are defined as in Proposition 
\ref{propint}.
\end{proposition}

\section{Modelling density-dependent cross-diffusion}
\label{secmodelling}

We now turn our attention to specific choices of the turning kernel and the turning frequency that lead to 
the 
special form of equations \eqref{fullsyst}. In the case of the turning kernel, we simply suppose that it is 
unperturbed from a 
normalized turning kernel $T_0$ which is a multiple of the identity,
\begin{equation}
\label{defT0}
T = T_0(\bv,\bv') := \frac{1}{|V|},
\end{equation}
corresponding to a uniform reorientation in velocity space. In this case, the turning kernel is unbiased and 
it clearly satisfies assumptions \eqref{T1} - \eqref{T4}. We also define the set $V$ of admissible velocities 
to be
\begin{equation}
\label{defV}
V := s \mathbb{S}^{n-1} := \{ s\bv \, : \, |\bv| = 1, \bv \in \mathbb{R}^ n \},
\end{equation}
for a certain constant speed $s > 0$. This choice is tantamount to a uniform distribution of velocities in 
all directions with constant magnitude $s > 0$. Clearly, $V$ is compact and symmetric with respect to the 
origin.

\subsection{Choice of the perturbed turning rate}

Regarding the turning frequency, we follow a suggestion by Othmer \textit{et al.} \cite{ODA88}. They observed 
that dispersal or aggregation can be modelled if one assumes that $\lambda$ depends explicitly on the 
marginal density $\rho$ (see also \cite{MCPF12}). Aggregation effects, for instance, can be described if 
$\lambda$ is a decreasing function of $\rho$, in which case the waiting time between jumps increases with the 
density. Likewise, $\lambda$ may also depend on the chemical concentration and its gradient. Experiments 
show, for instance, that the turning time is low when the concentration of the chemoattractant is low (see, 
e.g., \cite{MRFSASS17} in the case of \textit{B. subtilis} and oxygen as chemoattractant\footnote{we are 
unaware of similar experimental results for \textit{E. coli}, for example, or for any other chemoattractant, 
such as peptone (nutrient); still, we conjecture that this assumption is reasonable within certain regimes of 
concentration.}). In view of the above observations, we propose the following space- and time-dependent 
unperturbed turning rate in \eqref{ptr},
\[
\lambda_0 (\xi,\tau) := \frac{\mu_0}{\rho(\xi,\tau) S(\xi,\tau)},
\]
where $\mu_0 > 0$ is a constant with the appropriate units: for $\lambda_0$ to have frequency units 
($1/[T]$), $\mu_0$ must have units of square moles over time ($[M]^2/[T]$), and it can be regarded as a 
turning frequency per square unit mass. To encode the dependence of $\lambda$ on the 
gradient of the chemical signal, we follow \cite{OtHi02} and set the $O(\epsilon)$-perturbation of the 
turning rate to be
\begin{equation}
\label{lambdaone}
\lambda_1 (\bv, \hat{S}) := \kappa(S) (\bv \cdot \nabla_\xi S),
\end{equation}
where $\kappa = \kappa(S)$ is a scalar function to be chosen (and interpreted) later. In this fashion, the 
turning rate takes the following form,
\begin{equation}
\label{fullam}
\lambda(\bv, \hat{S}) = \lambda_0(\xi,\tau) + \ep \lambda_1(\bv,\hat{S}) = 
\frac{\mu_0}{\rho(\xi,\tau) S(\xi,\tau)} + \ep \kappa(S(\xi,\tau)) \, (\bv \cdot \nabla_\xi S(\xi,\tau) ).
\end{equation}
\begin{remark}
The turning frequency considered here is of Schnitzer type \cite{Schn93}, $\lambda = \lambda_0 + \ep 
\lambda_1(\bv)$, in which the order $O(\ep)$ perturbation depends on the velocity, whereas $\lambda_0$ does 
not. $\lambda_0$ and $\lambda_1$, however, may depend on space and time (see also the Discussion section in 
\cite{OtHi02}). We assume that this dependence happens through the marginal population density $\rho$ and the 
chemical concentration density $S$. It is to be observed that our choice for such dependence, namely 
$\lambda_0 \sim \rho^{-1} S^{-1}$, is motivated by experimental observations which indicate that the running 
time is low when the chemical concentration is low (cf. \cite{MRFSASS17}), and that $\lambda$ must be a 
decreasing function of $\rho$ in order to model aggregation \cite{MCPF12,ODA88}. Moreover, notice that the 
dependence of $\lambda$ on space and time occurs at the lowest order in $\epsilon$, affecting the form of the 
resulting chemotactic velocity. 
\end{remark}

If the turning kernel $T = T_0$ is defined by \eqref{defT0} then the pseudo-inverse, $\mathcal{F}_0 \, : \, 
\langle 1 \rangle^\perp \subset L^2(V) \to L^2(V)$, of the operator $\mathcal{L}_0$ defined in \eqref{defL0} 
reduces to multiplication by $- \lambda_0^{-1}$ (see \cite{HiOt00}). Therefore, the expression for the 
diffusion tensor \eqref{defD} reduces to
\[
\bD = - \frac{1}{|V|} \int_V \bv \otimes (\cF_0 \bv)  \, d\bv = \frac{\lambda_0^{-1}}{|V|} \int_V \bv \otimes 
\bv 
\, d\bv.
\]
Since $|V| = s^{n-1} |\mathbb{S}^{n-1}| =: s^{n-1} \omega_n$ and 
\[
\int_V \bv \otimes \bv \, d\bv = \int_{\mathbb{S}^{n-1}} s^2 \boldsymbol{\eta} \otimes \boldsymbol{\eta} 
s^{n-1} \, dS_{\boldsymbol{\eta}} = s^{n+1} \frac{\omega_n}{n} \mathbb{I}_n,
\]
then the (cross) diffusion tensor takes the form $\bD = (\lambda_0 s^2 / n) \mathbb{I}_n$, that is,
\begin{equation}
\label{crossD}
\bD = \bD (\rho,S) = \left(\frac{s^2}{\mu_0 n} \right) \rho S \, \mathbb{I}_n.
\end{equation}

In order to compute the chemotactic velocity \eqref{chemvel}, first we observe that the average bias vanishes,
\[
\begin{aligned}
\overline{\lambda_1}(\bv) = \int_V \lambda_1(\bv') T_0(\bv, \bv') \, d\bv' &= \frac{1}{|V|} \int_V \kappa(S) 
(\bv' \cdot \nabla_\xi S) \, d\bv' \\ &= \frac{\kappa(S)}{|V|} \left( \int_V \bv \, d\bv \right) \cdot 
\nabla_\xi S \\&= 0,
\end{aligned}
\]
inasmuch as $V$ is symmetric with respect to the origin and, consequently, $\int_V \bv \, d \bv = 0$. 
Therefore the chemotactic velocity \eqref{chemvel} reduces to
\[
\bw_c = \frac{1}{|V|} \int_V \bv \cF_0 \big( \lambda_1(\bv,\hat{S})  \big) \, d\bv = - \, 
\frac{\lambda_0^{-1}}{|V|} \int_V \bv \kappa(S) (\bv \cdot \nabla_\xi S) \, d\bv = - \mathbb{X} \nabla_\xi S, 
\]
where
\begin{equation}
\label{chemst}
\mathbb{X} := \frac{\kappa(S)}{\lambda_0 |V|} 
\int_V \bv \otimes \bv \, d\bv,
\end{equation}
is the \textit{chemotactic sensitivity tensor}. In view that $V = s \mathbb{S}^{n-1}$ we readily obtain
\begin{equation}
\label{ourchemst}
\mathbb{X}(\rho,S) = \left( \frac{s^2}{\mu_0 n}\right) \kappa(S) \rho S \,  \mathbb{I}_n.
\end{equation}
Therefore, the chemotactic velocity takes the form
\begin{equation}
\label{chemvel2}
\bw_c =  - \left( \frac{s^2}{\mu_0 n}\right) \kappa(S) \rho S \nabla_\xi S.
\end{equation}

The scalar function $\kappa = \kappa(S)$ can be interpreted as the standard \textit{chemotactic sensitivity 
function}. If $\kappa < 0$ then the chemotactic signal is attractive and the chemical is known as a 
chemo-attractant. If $\kappa > 0$ then the signal is repulsive and the substance is known as a 
chemo-repellent. As expected, the attractive/repellent behaviour of the chemical signal can be modelled 
by the choice of the sign of the function $\kappa$. Notice that, from the form of \eqref{lambdaone} and since 
$\lambda_1$ has frequency units, the function $\kappa$ must have dimensions of $1/[M]$ ($1/$ chemical 
concentration moles).

Upon substitution into the resulting equation \eqref{parabeq} for $\bpz = |V|^{-1} \rho$ we arrive at the 
following parabolic, cross-diffusion equation for the density function $\rho$:
\begin{equation}
\label{finaleq}
\frac{\partial \rho}{\partial \tau} = \nabla_{\xi} \cdot \left( \Big( \frac{s^2}{\mu_0 n} \Big) \, \rho \, S 
\,  
\nabla_\xi \rho \right) + \nabla_{\xi} \cdot \left( \Big( \frac{s^2}{\mu_0 n} \Big) \,  \rho^2 \, S \,  
\kappa(S) \nabla_\xi S \right).
\end{equation}

\begin{remark}
It should be observed that the quantity
\begin{equation}
\label{defcrossd}
D(\rho, S) := \Big( \frac{s^2}{\mu_0 n} \Big) \rho S,
\end{equation}
has dimensions of square length over time ($\,[L]^2/[T]$) and, thus, it may be interpreted as an effective 
diffusion coefficient for the population density $\rho$. Moreover, equation \eqref{finaleq} underlies the 
experimental (yet, phenomenological) observation by Ben-Jacob and his group \cite{GKCB,B-JCoLev00}. Indeed, 
if we define the chemotactic flux as
\[
\boldsymbol{J}_c := - \Big( \frac{s^2}{\mu_0 n} \Big) \kappa(S) \rho^2 S \nabla_\xi S,
\]
then it clearly follows the empirical rule \eqref{empiricalrule}, where the bacterial 
response function $\zeta$ is proportional to the product of $\rho$ with the effective diffusion coefficient,
\[
\zeta(\rho,S) = \Big( \frac{s^2}{\mu_0 n} \Big) \rho^2 S = \rho D(\rho,S).
\]
\end{remark}

Finally, as to the chemotactic sensitivity function is concerned, $\kappa = \kappa(S)$ will be assumed to be 
attractive ($\kappa(S) < 0$) and to follow the Lapidus-Schiller model or ``receptor law'' \cite{LapSch1},
\begin{equation}
\label{rlaw}
\kappa(S) = - \, \frac{\chi_0 K_d}{(K_d+S)^2},
\end{equation}
where $\chi_0>0$ is a dimensionless constant measuring the strength of the chemotaxis, and $K_d > 0$ is the 
receptor-ligand binding dissociation constant, which has nutrient concentration units $[M]$ (it represents 
the nutrient level needed for half receptor to be occupied), has a unique value that depends on the bacterial 
strain and the chemical signal under consideration, and must be determined experimentally. The 
receptor law was designed to fulfill experimental observations (see \cite{MOA}) which indicate that, for 
very high concentrations of the chemical signal, the chemotactic response of bacteria vanishes due to 
saturation of the receptors. We also recall that it has been reported from experiments (cf. Block \textit{et 
al.} \cite{BSB83}) that the turning frequency depends on the rate of change of the fraction $S/(K_d + S)$, 
and that is why the chemotactic sensitivity is encoded in the expression for the turning frequency.

Upon substitution of the receptor law \eqref{rlaw} into \eqref{finaleq} we arrive at the following 
cross-diffusion equation
\begin{equation}
 \label{finalrhoeq}
\frac{\partial \rho}{\partial \tau} = \nabla_{\xi} \cdot \left( \Big( \frac{s^2}{\mu_0 n} \Big) \, \rho \, S 
\,  \nabla_\xi \rho \right) - \nabla_{\xi} \cdot \left( \Big( \frac{s^2}{\mu_0 n} \Big) \,  \rho^2 \, S \,  
\frac{\chi_0 K_d}{(K_d+S)^2} \nabla_\xi S \right).
\end{equation}
Finally, the boundary condition \eqref{bcondi} takes the form
\begin{equation}
\label{finalbc}
\Big( \rho S \nabla_\xi \rho - \rho^2 S \frac{\chi_0 K_d}{(K_d + S)^2}  \nabla_\xi S \Big) \cdot 
\hat{\nu}(\xi) = 0, \qquad 
\xi \in \partial \widetilde{\Omega}, \; \tau > 0.
\end{equation}

\subsection{Equation for the chemical signal, reaction term and non-di\-men\-sio\-na\-li\-za\-tion}

The pure movement velocity jump process expressed by equation \eqref{wvjp} is Markovian, that is, the 
probability that agents reach the position $x$ at time $t$ does not depend on previous times, $t - T$, with 
$T > 0$. In other words, there is no history dependence. In our case, this property holds because the jumps 
are governed by a Poisson process (see, e.g., Feller \cite{Fell3ed}). Hence, it is legitimate to account for 
population growth/production by simply adding a reaction function to the right hand side of equation 
\eqref{wvjp}. Here we adopt the kinetics considered by Kawasaki \emph{et al.} \cite{KMMUS} and by Leyva 
\emph{et al.} \cite{LMP1}, for which the chemical signal is due to nutrient concentration. If 
$\widetilde{\cG}(p,S)$ represents the consumption rate of the nutrient by the bacteria, then the growth rate 
of the former is $\theta \widetilde{\cG}(p,S)$, where $\theta > 0$ is a (dimensionless) conversion constant 
or conversion rate of consumed nutrient to bacterial growth. This consumption rate function or kinetic 
function is chosen to follow the Michaelis-Menten rule,
\[
\widetilde{\cG}(p,S) = \frac{k p S}{1 + \gamma S},
\]
whereupon $k > 0$ is the intrinsic consumption rate and $\gamma > 0$ is a saturation constant ($k/\gamma$ is 
the maximum consumption rate by one single cell.) This function can be approximated by
\begin{equation}
\label{defkineticG}
\widetilde{\cG} (p,S) = k p S,
\end{equation}
for low nutrient concentrations (see \cite{KMMUS}). It measures the probability of encounters between 
bacterial agents with nutrient molecules. 

We now make one further assumption, and suppose that \emph{the number of directional changes outnumber the 
birth events}, implying that the pure movement velocity jump process occurs on a much faster scale than the 
production events due to kinetics (see, e.g., Hillen \cite{Hil03}). Since $t$ is the fast time variable and 
$\tau$ denotes the slow time scale, the pure kinetics are governed by an equation of the form $p_\tau = 
\widetilde{\cG}(p,S)$, that is, $p_t = \epsilon^2 \widetilde{\cG}(p,S)$. Hence we define,
\[
{\cG}(p,S) := \epsilon^2 \widetilde{\cG}(p,S) = \epsilon^2 k p S,
\]
and the resulting scaled velocity jump process equation with reaction term reads
\begin{equation}
\label{Rsvjp}
\begin{aligned}
\ep^2 \frac{\partial}{\partial \tau} p(\xi,\tau,\bv) + \ep \bv \cdot \nabla_\xi p(\xi,\tau,\bv) = &- \lambda 
p(\xi,\tau,\bv) + \int_V \lambda T(\bv,\bv') p(\xi,\tau,\bv') \, d\bv' + \\&+\epsilon^2 k p(\xi,\tau,\bv) 
S(\xi,\tau).
\end{aligned}
\end{equation}

It is to be observed that the kinetic term appears at order $O(\epsilon^2)$. If we perform the Hilbert 
expansion approximation on equation \eqref{Rsvjp} as in section \ref{secipl}, the kinetic term shows up at 
the right hand side of the order $O(\epsilon^2)$-equation \eqref{ordertwo}, in the form $k \theta 
\bpz(\xi,\tau) S(\xi,\tau)$. The solvability condition \eqref{solvcond2} is then modified to read
\[
0 = \int_V \Big[ \frac{\partial}{\partial \tau} \bpz(\xi,\tau) + \cR_1 \big( \cF_0 \big( 
\cR_1(\bpz(\xi,\tau)) 
\big) \big) + k \theta \bpz(\xi,\tau) S(\xi,\tau) \Big] \, d\bv.
\]
Since $\int_V k \theta \bpz(\xi,\tau) S(\xi,\tau) \, d\bv = |V| k \theta \bpz(\xi,\tau) S(\xi,\tau)$, then 
the resulting parabolic limit equa\-tion \eqref{parabeq} now incorporates the kinetic term as follows,
\[
\frac{\partial \bpz}{\partial \tau} = \nabla_{\xi} \cdot \big( \bD \, \nabla_\xi \bpz \big) - \nabla_{\xi} 
\cdot 
\big( \bpz \bw_c \big) + k \theta \bpz S.
\]
Specialized to the choices of the turning frequency and kernel in \eqref{fullam} and \eqref{defT0}, last 
equation leads to the following bacterial density equation with kinetic term,
\begin{equation}
\label{finaleq2}
 \frac{\partial \rho}{\partial \tau} = \nabla_{\xi} \cdot \Big( \sigma  \rho \, 
S \,  \nabla_\xi \rho \Big) -  \, \nabla_{\xi} \cdot \left( \sigma  \rho^2 S  \frac{\chi_0 K_d}{(K_d+S)^2} 
\nabla_\xi 
S \right) + k \theta \rho S,
\end{equation}
where
\begin{equation}
\label{defofsigma}
\sigma := \frac{s^2}{\mu_0 n}.
\end{equation}
\begin{remark}
If the medium is very hard then the typical velocity of agents $s$ is small, and the turning frequency 
per square unit mass $\mu_0$ is large (in soft, non-resistant media the bacteria tend to continue with 
same velocity in the absence of an external chemotactic signal). In this fashion, lower values of $\sigma$ 
indicate harder substrates. 
\end{remark}

In the same fashion, since the boundary condition for $\bpz$ is determined by the $O(\epsilon)$-no mass flux 
condition for $p_1$, the kinetic term plays no role and the boundary condition for $\rho$ is given by 
\eqref{finalbc} (details are left to the reader). Finally, the dynamics of the chemical (or nutrient) 
concentration $S = S(\xi,\tau)$ is governed by a standard reaction-diffusion equation (which is well-known to 
derive as the diffusion limit of a stochastic positional jump process) of the form
\begin{equation}
\label{eqforS}
\frac{\partial S}{\partial \tau} = D_S \Delta_\xi S - k \rho S, \qquad \xi \in \widetilde{\Omega}, \; \tau > 
0,
\end{equation}
and subject to Neumann (no-flux) boundary conditions,
\begin{equation}
\label{Sbc}
\nabla_\xi S \cdot \hat{\nu}(\xi) = 0, \qquad \xi \in \partial \widetilde{\Omega}, \, \tau > 0.
\end{equation}
Here $D_S > 0$ denotes the diffusion constant associated to the nutrient concentration and $- k \rho S$ is 
the consumption rate of the nutrient by the bacteria. The system of equations is endowed with initial 
conditions of the form 
\begin{equation}
\label{icrhoS}
\rho(\xi,0) = \rho_0(\xi), \quad S(\xi, 0) = S_0(\xi), \qquad \xi \in \widetilde{\Omega},
\end{equation}
being $\rho_0, S_0$, known functions. The resulting system is given by equation \eqref{finaleq2} for the 
bacterial density, equation \eqref{eqforS} for the chemical signal, the boundary conditions \eqref{finalbc} 
and \eqref{Sbc}, and the initial conditions \eqref{icrhoS}. It is clear that we have arrived at system of 
equations \eqref{fullsyst} - \eqref{incond} (after relabelling the variables for the bacterial and chemical 
concentration densities).

For analytical and numerical purposes we perform a further non-dimensio\-na\-li\-za\-tion procedure. We apply 
the same transformation as in \cite{LMP1}, and define, with a slight abuse of notation, the non-dimensional 
independent space and time variables,
\[
x := \left( \frac{\theta K_d k}{D_S}\right)^{1/2} \xi, \qquad t := \theta k K_d \tau,
\]
the chemical and agent densities,
\[
v := \frac{S}{K_d}, \qquad u := \frac{\rho}{\theta K_d},
\]
as well as the non-dimensional constants
\[
\sigma_0 := \left( \frac{\theta K_d^2}{D_S}\right) \sigma, \qquad \widetilde{s} := \frac{s}{(\theta k K_d 
D_S)^{1/2}}.
\]
Upon substitution we arrive at a system of equations for the rescaled bacterial density, $u = u(x,t)$, and 
the 
nutrient concentration, $v = v(x,t)$, namely,
\begin{equation}
\label{onondimsyst}
\begin{aligned}
u_t &=  \nabla \cdot ( \sigma_0 uv \nabla u) -  \nabla \cdot \left(\sigma_0 \chi_0  \frac{u^2 
v}{(1+v)^2}\nabla v \right) + uv,\\
v_t &= \Delta v - uv,
\end{aligned}
\qquad x \in \Omega, \; t > 0,
\end{equation}
with boundary conditions,
\begin{equation}
\label{onondimbc}
\begin{aligned}
\left(uv \nabla u - \chi_0  \frac{u^2 v}{(1+v)^2} \nabla v \right) \cdot \hat{\nu} = 0, \\
\nabla v \cdot \hat{\nu} = 0,
\end{aligned}
\qquad x \in \partial \Omega, \; t > 0,
\end{equation}
and initial conditions
\begin{equation}
\label{onondimic}
u(x,0) = u_0(x), \quad v(x,0) = v_0(x), \qquad x \in \Omega,
\end{equation}
where now $\Omega$ denotes the rescaled bounded, open domain in the new non-dimensional variables. 

\begin{remark}
Under the non-dimensionalization procedure considered here, system \eqref{onondimsyst} - \eqref{onondimic} 
contains the minimum number of free physical parameters, which in our case reduce to the chemotactic 
sensitivity $\chi_0 > 0$, a dimensionless constant coefficient measuring the intensity of the chemotactic 
signal, and $\sigma_0 > 0$, a non-dimensional parameter measuring the hardness of the substrate (agar).
\end{remark}

As a final observation, we notice that the chemotactic sensitivity function may be chosen freely, depending on
the species of bacteria and of the chemical under consideration. Even in the case of a single species like 
\emph{B. subtilis} one may consider (besides the Lapidus-Schiller receptor law) the classical 
Keller-Segel model \cite{KeSe1,KeSe2}, the Rivero-Tranquillo-Buettner-Lauffenberger (or RTBL) chemotactic 
flux \cite{RTBL89}, or the recently proposed finite range log-sensing model \cite{MRFSASS17}. They all 
appear indicated as the function $\kappa = \kappa(S)$ in both the density equation \eqref{finalrhoeq} and the 
boundary conditions \eqref{finalbc}. After the same non-dimensionalization procedure the resulting system of 
equations can be expressed more generally in terms of an arbitrary chemotactic sensitivity function $\chi = 
\chi(v)$. Recall that, for the chemotactic flux of the form $\boldsymbol{J}_c = - \zeta(u,v) \chi(v) \nabla 
v$ to be attractive, one requires the chemotactic sensitivity function to be negative, $\chi (\cdot) < 0$, 
provided, of course, that $\zeta (\cdot) > 0$. Thus, one may change its sign in the expressions with no harm. 
A more general result can be expressed in terms of the non-dimensional system as follows.

\begin{theorem}[formal parabolic limit]\label{mainprop}
Let $\Omega \subset \R^n$ be an open, bounded domain, with piecewise smooth boundary $\partial \Omega$. Then 
system of equations 
\begin{equation}
\label{nondimsyst}
\begin{aligned}
u_t &= \sigma_0 \nabla \cdot \left( uv \nabla u - u^2 v \chi(v) \nabla v \right) + uv,\\
v_t &= \Delta v - uv,
\end{aligned}
\qquad x \in \Omega, \; t > 0,
\end{equation}
together with boundary and initial conditions of the form
\begin{equation}
\label{nondimbc}
\begin{aligned}
\big(uv \nabla u -   u^2 v \chi(v)\nabla v\big) \cdot \hat{\nu} = 0, \\
\nabla v \cdot \hat{\nu} = 0,
\end{aligned}
\qquad x \in \partial \Omega, \; t > 0,
\end{equation}
and
\begin{equation}
\label{nondimic}
u(x,0) = u_0(x), \quad v(x,0) = v_0(x), \qquad x \in \Omega,
\end{equation}
respectively, constitute the formal parabolic limit when $\epsilon \to 0^+$ of the solutions to a velocity 
jump process governed by a transport equation for a non-dimensional agent distribution $q = q(x,t,\bv)$, of 
the form
\begin{equation}
\label{nondimtranseq}
\epsilon^2 q_t + \epsilon \bv \cdot \nabla q = - \lambda q + \int_{\overline{V}} \lambda T_0(\bv,\bv') q \, 
d\bv' + \epsilon^2 q v,
\end{equation}
and the standard reaction diffusion equation for the (non-dimensional) chemical signal concentration $v = 
v(x,t)$ in \eqref{nondimsyst}, where 
\[
u(x,t) = \int_{\overline{V}} q(x,y,\bv) \, d\bv,
\]
is the zero velocity moment or marginal density. Here 
\[
\overline{V} = \widetilde{s} \, \mathbb{S}^{n-1} := \{ \widetilde{s}\, \bv \, : \, |\bv| = 1, \bv \in 
\mathbb{R}^ n \},
\]
is the uniform set of non-dimensional velocities, the turning kernel and turning frequency admit asymptotic 
expansions of the form
\[
T_0(\bv, \bv') = \frac{1}{|\overline{V}|}, \qquad \lambda = \frac{1}{u(x,t) v(x,t)} - \epsilon \chi(v) (\bv 
\cdot \nabla v),
\]
and the solution $q$ to the transport equation \eqref{nondimtranseq} is further endowed with boundary 
conditions of no normal mass flux type satisfying
\[
\int_{\overline{V}} q(x,t,\bv) (\bv \cdot \hat{\nu}(x)) \, d\bv = 0, \qquad x \in \partial \Omega, \, t>0,
\]
together with standard Neumann boundary conditions for the chemical concentration as in \eqref{nondimbc}. The 
function $\chi = \chi(v)$ is understood as the chemotactic sensitivity function, and it is positive 
(respectively, negative) in the case of a chemoattractant (respectively, chemorepellent).
Finally, $\epsilon > 0$ is a dimensionless small parameter associated to the velocity jump process in the 
parabolic regime. The limiting process is taken in the following sense: if $q$ admits a formal Hilbert 
expansion of the form $q = q_0 + \epsilon q_1 + O(\epsilon^2)$ for $0 < \epsilon \ll 1$ small, then the 
leading order term is independent of the velocity $\bv \in \overline{V}$, it is proportional to the 
marginal density, $q_0 = |\overline{V}| u$, and $u$ and $v$ satisfy system
\eqref{nondimsyst} - \eqref{nondimic}. Moreover, $v_0(x) := v(x,0)$ denotes the initial nutrient 
concentration, 
and
\[
u_0(x) := u(x,0) = \int_{\overline{V}} q(x,0,\bv) \, d\bv,
\]
is the initial distribution of agents for all possible velocities.
\end{theorem}

\section{Discussion}
\label{secdiscuss}

In this paper, a formal macroscopic limit (in the parabolic re\-gi\-me) of a velocity jump 
process in order to arrive at systems of equations \eqref{fullsyst} - \eqref{incond} was performed. The 
system is endowed with a doubly degenerate cross-diffusion term and a nutrient taxis drift term. Both fluxes 
(the diffusive and the chemotactic) are density-dependent, degenerate, and related to each other. The 
derivation makes precise the microscopic interpretation of the phenomenological bacterial response function 
introduced by Ben-Jacob and his group in terms of the turning frequency at the level of the kinetic transport 
equation. The latter is supposed to be perturbed from a turning rate which does not depend on the velocity of 
the cells, but does depend on the chemical concentration. The first order perturbation of the turning 
frequency is assumed to be of Schnitzer type, that is, it depends on the velocity of the cells; in addition, 
it encodes the response to the chemotactic signal as well. The turning kernel corresponds to a uniform 
reorientation in velocity space. These choices for the turning frequency and kernel are compatible with the 
observed frequency of tumbles as a function of time, position and direction, keeping both the run speed and 
the distribution of changes in direction essentially constant. Moreover, the decreasing functional form of 
the unperturbed turning frequency in terms of the chemical concentration obeys recent experimental 
observations. It is also decreasing as a function of the bacterial density in order to model aggregation, as 
suggested by Othmer \emph{et al.} \cite{ODA88}. Its functional form is arbitrary and suitably chosen for our 
needs, but it may be of course modified to fit other observed phenomena, or different responses in other 
parameter regimes. In other words, the procedure presented here can be applied to more general situations.

Up to our knowledge this is the first derivation of a doubly degenerate cross-diffusion model with 
density-dependent (and related to the diffusion) chemotactic drift term from a velocity jump process. 
This derivation is formal, but helpful to understand the interplay of the microscopic description of both 
diffusion and taxis, as well as the main features of the mean field model. From the mathematical viewpoint, 
systems of the form \eqref{fullsyst} pose technical challenges which are of interest to the community working 
on partial differential equations, such as the existence and regularity of global solutions (for recent 
advances, see \cite{DaiDu16a,PlWi17a,Win14b,Win17pre}). Another open mathematical question is the rigorous 
scaling limit, in the spirit of the analysis by Chalub \emph{et al.} \cite{CMPS04,CDMOSS06}, of the solutions 
to the transport equation converging to the weak solutions of the macroscopic diffusion-drift system 
recently obtained in \cite{PlWi17a}. This is an interesting topic for further investigation. Finally, from 
the biological viewpoint, it is of course of interest to explore the consequences of allowing turning 
frequencies to depend on density distributions in order to incorporate more complex spatial and temporal 
dependencies of microscopic quantities into models of related phenomena, such as macromolecular crowding 
\cite{Ellis01}, aggregation \cite{MCPF12}, or microscopic exclusion processes of intracellular transport 
\cite{GFP16}, just to mention a few.

\section*{Acknowledgements}

The author warmly thanks Thomas Hillen and Michael Winkler for enlightening conversations. This research was 
partially supported by DGAPA-UNAM, program PAPIME, grant PE-104116.

\def\cprime{$'$}




\end{document}